\documentclass[12pt,leqno]{amsart}
\usepackage{amsmath}
\usepackage{amssymb}
\def\overset#1#2{{\mathrel{\mathop {{#2}_{}}\limits^{#1}}}}
\def\underset#1#2{{\mathrel{\mathop {{}_{} {#2}}\limits_{{#1}_{}}}}}
\def\upplim_#1{\underset{#1}{\overline\lim}\;}
\def\lowlim_#1{\underset{#1}{\underline\lim}\;}
\newtheorem{corollary}[equation]{Corollary}
\newtheorem{definition}[equation]{\indent{\it Definition}\rm }
\newtheorem{claim}[equation]{\indent{\it Claim}\rm }

\newtheorem{lemma}[equation]{Lemma}
\newtheorem{proposition}[equation]{Proposition}

\newtheorem{theorem}[equation]{Theorem}

\newcommand{\B}{\mathbb{B}}
\newcommand{\C}{{\mathbb{C}}}

\renewcommand{\P}{{\mathbb{P}}}

\newcommand{\R}{{\mathbf{R}}}
\newcommand{\ddc}{\mathrm{dd}^c}
\newcommand{\rank}{\mathrm{rank}}

\newcommand{\ric}{\mathrm{Ric}}

\newcommand{\Z}{\mathbf{Z}}

\numberwithin{equation}{section}

\title[Non-integrated defect relation meromorphic maps]{Non-integrated defect relation for meromorphic maps from a K\"{a}hler manifold intersecting hypersurfaces in subgeneral of $\P^n(\C)$} 

\author{Si Duc Quang}
\author{Nguyen Thi Quynh Phuong}
\author{Nguyen Thi Nhung}

\begin{document}

\begin{abstract}
In this article, we establish a truncated non-integrated defect relation for meromorphic mappings from an $m$-dimensional complete K\"{a}hler manifold into $\P^n(\C)$  intersecting $q$ hypersurfaces $Q_1,...,Q_q$ in $k$-subgeneral position of degree $d_i$, i.e., the intersection of any $k+1$ hypersurfaces is emptyset. We will prove that
$$ \sum_{i=1}^q\delta_f^{[u-1]}(Q_i)\le (k-n+1)(n+1)+\epsilon+\frac{\rho u(u-1)}{d},  $$
where $u$ is explicitly estimated and $d$ is the least common multiple of $d_i'$s. Our result generalizes and improves previous results. In the last part of this paper we will apply this result to study the distribution of the Gauss map of minimal surfaces.
\end{abstract}

\maketitle

\def\thefootnote{\empty}

\section{Introduction and Main result} 

Let $M$ be a complete K\"{a}hler manifold of dimension $m$. Let $f: M\longrightarrow \P^n(\C)$ be a meromorphic mapping and $\Omega_f$ be the pull-back of the Fubini-Study form $\Omega$ on $\P^n(\C)$ by $f$. For a positive integer  $\mu_0$ and a hypersurface $D$ of degree $d$ in $\P^n(\C)$ with $f(M)\not\subset D$, we denote by $\nu_f(D)(p)$ the intersection multiplicity of the image of $f$ and $D$ at $f(p)$. 

In 1985, H. Fujimoto \cite{F85} defined the notion of the non-integrated defect of $f$ with respect to $D$ truncated to level $\mu_0$ by
$$\delta_f^{[\mu_0]}:= 1- \inf\{\eta\ge 0: \eta \text{ satisfies condition }(*)\}.$$
Here, the condition (*) means that there exists a bounded non-negative continuous function $h$ on $M$ whose order of each zero is not less than $\min \{\nu_f(D), \mu_0\}$ such that
$$d\eta\Omega_f +\dfrac{\sqrt{-1}}{2\pi}\partial\bar\partial\log h^2\ge [\min\{\nu_f(D), \mu_0\}].$$
And then he gave a result analogous to the defect relation in Nevanlinna theory as follows. 

\vskip0.2cm
\noindent
{\bf Theorem  A} (see \cite[Theorem 1.1]{F85}). \ {\it Let $M$ be an $m$-dimensional complete K\"ahler manifold  and $\omega$ be a K\"ahler form of $M.$  Assume that the universal covering of $M$ is biholomorphic to a ball in $\mathbb C^m.$ Let $f : M\to \P^n(\C)$ be a meromorphic map which is linearly nondegenerate (i.e., its image is not contained in any hyperplane of $\P^n(\C)$). Let $H_1,\cdots,H_q$ be hyperplanes of $\P^n(\C)$ in general position. For some $\rho\ge 0$, if there exists a bounded continuous function $h\ge 0$ on $M$ such that 
$$\rho \Omega_f +\ddc \log h^2\ge \ric \ \omega,$$
then 
$$\sum_{i=1}^{q} \delta_{f}^{[n]} (H_i) \le n+1+ \rho n(n+1).$$}

Recently,  M. Ru-S. Sogome \cite{RS} generalized Theorem A to the case of meromorphic mappings intersecting a family of hypersurfaces in general position. After that, Q. Yan \cite{Y} extended Theorem A by consider the case where the family of hypersurfaces in subgeneral position. He proved the following. 

\vskip0.2cm
\noindent
{\bf Theorem  B} (see \cite[Theorem 1.1]{Y}). \ {\it Let $M$ be an $m$-dimensional complete K\"ahler manifold  and $\omega$ be a K\"ahler form of $M.$  Assume that the universal covering of $M$ is biholomorphic to a ball in $\mathbb C^m.$
Let $f$ be an algebraically nondegenerate meromorphic map of $M$ into $\P^n(\C)$. 
Let $Q_1,...,Q_q$ be hypersurfaces in $\P^n(\C)$ of degree $P_{Ij},$ in $k$-subgeneral position in $\P^n(\C).$  
Let $d = l.c.m.\{Q_1,...,Q_q\}$ (the least common multiple of $\{Q_1,...,Q_q\}$).
Denote by $\Omega_f$ the pull-back of the Fubini-Study form of  $\P^n(\C)$ by $f.$
Assume that for some $\rho \ge 0,$ there exists a bounded continuous function $h \geq 0$ on $M$ such that
$$\rho\Omega_f + \ddc \log h^2 \geq \ric \ \omega.$$
Then, for each  $\epsilon>0,$ we have
$$\sum_{j=1}^q {\delta}_{f}^{[u-1]} (Q_j ) \le k(n+1)+ \epsilon +\dfrac{\rho u(u-1)}{d},$$
where $u= \bigl ( ^{N +n}_{\ \ n}\bigl )\leq (3ekd I(\epsilon^{-1}))^n(n+1)^{3n}$ and
$N=2kdn^2(n+1)^2 I(\epsilon^{-1}).$}

Here, for a real number $x$, we define $I(x):=\min\{a\in\mathbf Z\ ;\ a>x\}$.
 
However, the above result of Q. Yan does not yet completely extend the results of H. Fujimoto and M. Ru-S. Sogome. Indeed, when the family of hypersurfaces in general position, i.e., $k=n$, the first term in the right hand side of the defect relation inequality is $n(n+1)$, which is bigger than $(n+1)$ as usual. Recently, T. V. Tan and V. V. Truong in \cite{TT} also gave a non-integrated defect relation for the family of hypersurfaces in subgeneral position, where this term is equal to $n+1$. But their definition of ``subgeneral position'' is quite special, which has an extra condition on the intersection of these $q$ hypersurfaces (see Definition 1.1(ii) in \cite{TT})

The first aim of this paper is to establish a non-integrated defect relation for meromorphic mappings of complete K\"{a}hler manifolds into $\P^n(\C)$ sharing hypersurfaces located in subgeneral position which generalizes the above mentioned results and improves the result of Q. Yan. In usual principle, to treat with the case of family of hypersurfaces in subgeneral position, we need to generalize the notion of Nochka weights. However for the case of hypersurfaces, there is no Nochka weights constructed. In order to over come this difficult, we will use a technique ``replacing hypersurfaces'' proposed in \cite{Q16-1,Q16-2}. Before stating our result, we recall the following.

Let $k\geq n$ and $q\geq k+1.$ Let $Q_1,...,Q_q$ be hypersurfaces in $\mathbb P^n(\mathbb C).$ The hypersurfaces $Q_1,...,Q_q$ are said to be in $k$-subgeneral position in $\mathbb P^n(\mathbb C)$ if 
$$Q_{j_1}\cap\cdots\cap Q_{j_{k+1}}=\emptyset\text{ for every }1\leq j_1<\cdots<j_{k+1}\leq q.$$
If  $\{Q_i\}_{i=1}^q$ is in $n$-subgeneral position then we say that it is in \textit{general position}.

Our main result is stated as follows.
 
\begin{theorem}\label{1.1} 
Let $M$ be an $m$-dimensional complete K\"ahler manifold  and $\omega$ be a K\"ahler form of $M.$  Assume that the universal covering of $M$ is biholomorphic to a ball in $\mathbb C^m.$ Let $f$ be an algebraically nondegenerate meromorphic map of $M$ into $\P^n(\C)$. Let $Q_1,...,Q_q$ be hypersurfaces in $\P^n(\C)$ of degree $d_j,$ in $k$-subgeneral position in $\P^n(\C).$  Let $d = l.c.m.\{d_1,...,d_q\}$ (the least common multiple of $\{d_1,...,d_q\}$). Assume that for some $\rho \ge 0,$ there exists a bounded continuous function $h \geq 0$ on $M$ such that
$$\rho\Omega_f + \ddc \log h^2 \geq \ric \ \omega.$$
Then, for each  $\epsilon>0,$ we have
$$\sum_{j=1}^q {\delta}^{[u-1]}_{f} (Q_j) \le p(n+1)+ \epsilon +\dfrac{\rho u(u-1)}{d},$$
where $p=k-n+1,\ u= \bigl ( ^{N +n}_{\ \ n}\bigl )\leq e^{n+2}(dp(n+1)^2 I(\epsilon^{-1}))^n$ and
$N=(n+1)d+p(n+1)^3I(\epsilon^{-1})d.$
\end{theorem}

Then we see that, if the family of hypersurfaces is in general position, i.e., $k=n$, then our result deduces the results of H. Fujimoto and also of M. Ru-S. Sogome. Of course, compaired to the original form of Cartan-Nochka's theorem where the first term in the right hand side of the defect relation inequality is $(2k-n+1)$, our result is still not yet optimal. Therefore, how to give a sharp defect relation in this case is an open question.

In the above theorem, letting $\epsilon =1+\epsilon'$ with $\epsilon'>0$ and then letting $\epsilon'\longrightarrow 0$, we obtain the following corollary.
\begin{corollary} With the assumption of Theorem \ref{1.1}, we have
$$\sum_{j=1}^q {\delta}^{[u-1]}_{f} (Q_j) \le p(n+1)+ 1+\dfrac{\rho u(u-1)}{d},$$
where $p=k-n+1,\ u= \bigl ( ^{N +n}_{\ \ n}\bigl )\leq e^{n+2}(dp(n+1)^2)^n$ and
$N=(n+1)d(1+p(n+1)^2).$
\end{corollary}
In the last part of this paper, we will apply Theorem \ref{1.1} to give a non-integrated defect relation of the Gauss map of a regular submanifold of $\C^m$ (see Theorem \ref{6.1.} below).

\section{Basic notions and auxiliary results from Nevanlinna theory}

\noindent
{\bf 2.1. Counting function.}\ We set $||z|| = \big(|z_1|^2 + \dots + |z_m|^2\big)^{1/2}$ for
$z = (z_1,\dots,z_m) \in \mathbb C^m$ and define
\begin{align*}
\B^m(r) &:= \{ z \in \mathbb C^m : ||z|| < r\},\\
S(r) &:= \{ z \in \mathbb C^m : ||z|| = r\}\ (0<r<\infty).
\end{align*}
Define 
$$v_{m-1}(z) := \big(dd^c ||z||^2\big)^{m-1}\quad \quad \text{and}$$
$$\sigma_m(z):= d^c \text{log}||z||^2 \land \big(dd^c \text{log}||z||^2\big)^{m-1}
 \text{on} \quad \mathbb C^m \setminus \{0\}.$$

 For a divisor $\nu$ on  a ball $\B^m(R)$ of $\mathbb C^m$, with $R>0$, and for a positive integer $M$ or $M= \infty$, we define the counting function of $\nu$ by
\begin{align*}
&\nu^{[M]}(z)=\min\ \{M,\nu(z)\},\\
&n(t) =
\begin{cases}
\int\limits_{|\nu|\,\cap \B^m(t)}
\nu(z) v_{m-1} & \text  { if } m \geq 2,\\
\sum\limits_{|z|\leq t} \nu (z) & \text { if }  m=1. 
\end{cases}
\end{align*}
Similarly, we define \quad $n^{[M]}(t).$

Define
$$ N(r,r_0,\nu)=\int\limits_{r_0}^r \dfrac {n(t)}{t^{2m-1}}dt \quad (0<r_0<r<R).$$

Similarly, define  \ $N(r,r_0,\nu^{[M]})$ and denote it by \ $N^{[M]}(r,r_0,\nu)$.

Let $\varphi : \mathbb C^m \longrightarrow \B^m(r) $ be a meromorphic function. Denote by $\nu_\varphi$ the zero divisor of $\varphi$. Define
$$N_{\varphi}(r,r_0)=N(r,r_0,\nu_{\varphi}), \ N_{\varphi}^{[M]}(r,r_0)=N^{[M]}(r,r_0,\nu_{\varphi}).$$

For brevity, we will omit the character $^{[M]}$ if $M=\infty$.

\vskip0.2cm
\noindent
{\bf 2.2. Characteristic function and first main theorem.}\ Let $f : \mathbb \B^m(R) \longrightarrow \mathbb P^n(\mathbb C)$ be a meromorphic mapping.
For arbitrarily fixed homogeneous coordinates $(w_0 : \dots : w_n)$ on $\mathbb P^n(\mathbb C)$, we take a reduced representation $\tilde f = (f_0 , \ldots , f_n)$, which means that each $f_i$ is a holomorphic function on $\B^m(R)$ and  $f(z) = \big(f_0(z) : \dots : f_n(z)\big)$ outside the analytic subset $\{ f_0 = \dots = f_n= 0\}$ of codimension $\geq 2$. Set $\Vert \tilde f \Vert = \big(|f_0|^2 + \dots + |f_n|^2\big)^{1/2}$.

The characteristic function of $f$ is defined by 
$$ T_f(r,r_0)=\int_{r_0}^r\dfrac{dt}{t^{2m-1}}\int\limits_{\B^m(t)}f^*\Omega\wedge v^{m-1}, \ (0<r_0<r<R). $$
By Jensen's formula, we will have
\begin{align*}
T_f(r,r_0)= \int\limits_{S(r)} \log\Vert f \Vert \sigma_m -
\int\limits_{S(r_0)}\log\Vert \tilde f\Vert \sigma_m +O(1), \text{ (as $r\rightarrow R$)}.
\end{align*}

Let $Q$ be a hypersurface in $\P^n(\C)$ of degree $d$. Throughout this paper, we sometimes identify a hypersurface with the defining polynomial if there is no confusion.  Then we may write
$$ Q(\omega)=\sum_{I\in\mathcal T_d}a_I\omega^I, $$
where $\mathcal T_d=\{(i_0,...,i_n)\in\mathbb Z_+^{n+1}\ ;\ i_0+\cdots +i_n=d\}$, $\omega =(\omega_0,...,\omega_n)$, $\omega^I=\omega_0^{i_0}...\omega_n^{i_n}$ with $I=(i_0,...,i_n)\in\mathcal T_d$ and $a_I\ (I\in\mathcal T_d)$ are constants, not all zeros. In the case $d=1$, we call $Q$ a hyperplane of $\P^n(\C)$.

The proximity function of $f$ with respect to $Q$, denoted by $m_f (r,r_0,Q)$, is defined by
$$m_f (r,r_0,Q)=\int_{S(r)}\log\frac{||\tilde f||^d}{|Q(\tilde f)|}\sigma_m-\int_{S(r_0)}\log\frac{||\tilde f||^d}{|Q(\tilde f)|}\sigma_m,$$
where $Q(\tilde f)=Q(f_0,...,f_n)$. This definition is independent of the choice of the reduced representation of $f$. 

We denote by $f^*Q$ the pullback of the divisor $Q$ by $f$. We may see that $f^*Q$ identifies with the zero divisor $\nu^0_{Q(\tilde f)}$ of the function $Q(\tilde f)$. By Jensen's formula, we have
$$N(r,r_0,f^*Q)=N_{Q(\tilde f)}(r,r_0)=\int_{S(r)}\log |Q(\tilde f)|\sigma_m-\int_{S(r_0)}\log |Q(\tilde f)| \sigma_m.$$
Then the first main theorem in Nevanlinna theory for meromorphic mappings and hypersurfaces is stated as follows.

\begin{theorem}[First Main Theorem] Let $f : \mathbb B^m(R) \to \P^n(\C)$ be a holomorphic map, and let $Q$ be a hypersurface in $\P^n(\C)$ of degree $d$. If $f (\mathbb C) \not \subset Q$, then for every real number $r$ with $r_0 < r < R$,
$$dT_f (r,r_0)=m_f (r,r_0,Q) + N(r,r_0,f^*Q)+O(1),$$
where $O(1)$ is a constant independent of $r$.
\end{theorem}

If $ \lim\limits_{r\rightarrow 1}\sup\dfrac{T(r,r_0)}{\log 1/(1-r)}= \infty,$ then the Nevanlinna's defect of $f$ with respect to the hypersurface $Q$ truncated to level $l$ is defined by
$$ \delta^{[l]}_{f,*}(Q)=1-\lim\mathrm{sup}\dfrac{N^{l}(r,r_0,f^*Q)}{T_f(r,r_0)}.$$
There is a fact that 
$$0\le \delta^{[l]}_f(Q)\le\delta^{[l]}_{f,*}(Q)\le 1. $$
(See Proposition 2.1 in \cite{RS})

\vskip0.2cm
\noindent
{\bf 2.3. Auxiliary results.}\ Repeating the argument in \cite[Proposition 4.5]{F85}, we have the following.

\begin{proposition}
Let $F_0,\ldots ,F_{N}$ be meromorphic functions on the ball $\B^m(R_0)$ in $\mathbb C^m$ such that $\{F_0,\ldots ,F_{N}\}$ are  linearly independent over $\mathbb C.$ Then  there exists an admissible set  
$$\{\alpha_i=(\alpha_{i1},...,\alpha_{im})\}_{i=0}^{N} \subset \mathbb Z^m_+$$
with $|\alpha_i|=\sum_{j=1}^{m}|\alpha_{ij}|\le i \ (0\le i \le N)$ such that the following are satisfied:

(i)\  $W_{\alpha_0,\ldots ,\alpha_{N}}(F_0,\ldots ,F_{N})\overset{Def}{:=}\det{({\mathcal D}^{\alpha_i}\ F_j)_{0\le i,j\le N}}\not\equiv 0.$ 

(ii) $W_{\alpha_0,\ldots ,\alpha_{N}}(hF_0,\ldots ,hF_{N})=h^{N+1}W_{\alpha_0,\ldots ,\alpha_{N}}(F_0,\ldots ,F_{N})$ for any nonzero meromorphic function $h$ on $\B^m(R_0).$
\end{proposition}

In \cite{RS}, M. Ru and S. Sogome gave the following lemma on logarithmic derivative for the meromorphic mappings of a ball in $\C^m$ into $\P^n(\C)$.

\begin{proposition}[{see \cite[Proposition 3.3]{RS}}]\label{pro2.2}
 Let $L_0,\ldots ,L_{N}$ be linear forms of $N+1$ variables and assume that they are linearly independent. Let $F$ be a meromorphic mapping of the ball $\B^m(R_0)\subset\C^m$ into $\P^{N}(\C)$ with a reduced representation $\tilde F=(F_0,\ldots ,F_{N})$ and let $(\alpha_0,\ldots ,\alpha_N)$ be an admissible set of $F$. Set $l=|\alpha_0|+\cdots +|\alpha_N|$ and take $t,p$ with $0< tl< p<1$. Then, for $0 < r_0 < R_0,$ there exists a positive constant $K$ such that for $r_0 < r < R < R_0$,
$$\int_{S(r)}\biggl |z^{\alpha_0+\cdots +\alpha_N}\dfrac{W_{\alpha_0,\ldots ,\alpha_N}(F_0,\ldots ,F_{N})}{L_0(\tilde F)\ldots L_{N}(\tilde F)}\biggl |^t\sigma_m\le K\biggl (\dfrac{R^{2m-1}}{R-r}T_F(R,r_0)\biggl )^p.$$
\end{proposition}
Here $z^{\alpha_i}=z_1^{\alpha_{i1}}...z_m^{\alpha_{im}},$ where $\alpha_i=(\alpha_{i1},...,\alpha_{im})\in\mathbb N^m_0$.

\section{Non-integrated defect relation for nondegenerate mappings sharing hypersurfaces in subgeneral position}

First of all, we need the following lemma due to \cite{Q16-1,Q16-2}. For the sake of completeness, we also include the proofs.

\begin{lemma}[{see \cite[Lemma 3.1]{Q16-1}, \cite[Lemma 3.1]{Q16-2}}]\label{lem3.1}
Let $Q_1,...,Q_{k+1}$ be hypersurfaces in $\P^n(\C)$ of the same degree $d\ge 1$, such that
$$\left (\bigcap_{i=1}^{k+1}Q_i\right )=\emptyset.$$
Then there exist $n$ hypersurfaces $P_{2},...,P_{n+1}$ of the forms
$$P_t=\sum_{j=2}^{k-n+t}c_{tj}Q_j, \ c_{tj}\in\C,\ t=2,...,n+1,$$
such that $\left (\bigcap_{t=1}^{n+1}P_t\right )=\emptyset,$ where $P_1=Q_1$.
\end{lemma}
\begin{proof} 
Set $P_1=Q_1$. It is easy to see that
$$ \dim \left(\bigcap_{i=1}^tQ_i\right)\le k-t+1,\ t=k-n+2,...,k+1,$$
where $\dim\emptyset =-\infty$.

Step 1. We firstly construct $P_2$ as follows. For each irreducible component $I$ of dimension $n-1$ of $P_1$, we put 
$$V_{1I}=\{c=(c_2,...,c_{k-n+2})\in\C^{k-n+1}\ ;\ I\subset Q_c,\text{ where }Q_c=\sum_{j=2}^{k-n+2}c_jQ_j\}.$$
Here, we also consider the case where $Q_c$ may be zero polynomial and it determines all $\P^n(\C)$. It easy to see that $V_{1I}$ is a subspace of $\C^{k-n+1}$. Since $\dim \left(\bigcap_{i=0}^{k-n+1}Q_i\right)\le k-2$, there exists $i (1\le i\le k-n+1)$ such that $I\not\subset Q_i$. This implies that $V_{1I}$ is a proper subspace of $\C^{k-n+1}$. Since the set of irreducible components of dimension $k-1$ of $P_0$ is finite, 
$$ \C^{k-n+1}\setminus\bigcup_{I}V_{1I}\ne\emptyset. $$
Then, there exists $(c_{12},...,c_{1(k-n+2)})\in\C^{k-n+1}$ such that the hypersurface 
$$P_2=\sum_{j=2}^{k-n+2}c_{1j}Q_j$$
does not contain any irreducible component of dimension $k-1$ of $P_1$.
This implies that $\dim \left(P_1\cap P_2\right)\le k-2.$

Step 2. Similarly, for each irreducible component $I'$ of dimension $n-2$ of $\left(P_1\cap P_2\right)$, put 
$$V_{2I'}=\{c=(c_2,...,c_{k-n+3})\in\C^{k-n+2}\ ;\ I'\subset Q'_c,\text{ where }Q'_c=\sum_{j=2}^{k-n+3}c_jQ_j\}.$$
Hence, $V_{2I'}$ is a subspace of $\C^{k-n+2}$. Since $\dim \left(\bigcap_{i=1}^{k-n+3}Q_i\right)\le n-3$, there exists $i, (2\le i\le k-n+3)$ such that $I'\not\subset Q_i$. Hence $V_{2I'}$ is a proper subspace of $\C^{k-n+2}$. Since the set of irreducible components of dimension $n-2$ of $\left(P_1\cap P_2\right)$ is infinite, 
$$ \C^{k-n+2}\setminus\bigcup_{I'}V_{2I'}\ne\emptyset. $$
Then, there exists $(c_{22},...,c_{2(N-k+3)})\in\C^{k-n+2}$ such that the hypersurface 
$$P_3=\sum_{j=2}^{k-n+3}c_{2j}Q_j$$
does not contain any irreducible components of dimension $n-2$ of $P_1\cap P_2$. Hence $\dim \left(P_1\cap P_2\cap P_3\right)\le n-3.$

Repeating again the above steps, after the $n$-th step we get the hypersurfaces $P_2,...,P_{n+1}$ satisfying that
$$ \dim\left(\bigcap_{j=1}^tP_j\right)\le n-t. $$
In particular, $\left(\bigcap_{j=1}^{n+1}P_j\right)=\emptyset.$ The lemma is proved.
\end{proof}

Let $f:M\longrightarrow \mathbb P^n(\mathbb C)$ be a meromorphic mapping with a reduced representation $\tilde f=(f_0,\ldots ,f_n)$. We define
$$ Q_i(\tilde f)=\sum_{I\in\mathcal I_d}a_{iI}f^I ,$$
where $f^I=f_0^{i_0}\cdots f_n^{i_n}$ for $I=(i_0,...,i_n)$. Then we can consider $f^*Q_i=\nu_{Q_i(\tilde f)}$ as divisors. We now have the following.

\begin{lemma}\label{3.2}
Let $\{Q_i\}_{i\in R}$ be a family of hypersurfaces in $\mathbb P^n(\mathbb C)$ of the common degree $d$ and let $f$ be a meromorphic mapping of $\mathbb C^m$ into $\mathbb P^n(\mathbb C)$. Assume that $\bigcap_{i\in R}Q_i=\emptyset$. Then, there exist positive constants $\alpha$ and $\beta$ such that
$$\alpha ||\tilde f||^d \le  \max_{i\in R}|Q_i(\tilde f)|\le \beta ||\tilde f||^d.$$
\end{lemma} 
\begin{proof}  
Let $(x_0:\cdots: x_n)$ be homogeneous coordinates of $\mathbb P^n(\mathbb C)$. Assume that each $Q_i$ is defined by $\sum_{I\in\mathcal I_d}a_{iI}x^I=0.$ 

Set $Q_i(x)=\sum_{I\in\mathcal I_d}a_{iI}x^I$ and consider the following function
$$ h(x)=\dfrac{\max_{i\in R}|Q_i(x)|}{||x||^d}, $$
where $||x||=(\sum_{i=0}^n|x_i|^2)^{\frac{1}{2}}$.

Since the function $h$ is positive continuous on $\P^n(\C),$ by the compactness of  $\P^n(\C)$, there exist positive constants $\alpha$ and $\beta$ such that $\alpha =\min_{x\in \mathbb P^n(\mathbb C)}h(x)$ and $\beta =\max_{x\in \mathbb P^n(\mathbb C)}h(x)$. Therefore, we have
$$\alpha ||\tilde f||^d \le  \max_{i\in R}|Q_i(\tilde f)|\le \beta ||\tilde f||^d.$$
The lemma is proved. 
\end{proof}

By Jensen's formula, we have the following lemma.
\begin{lemma}\label{3.3b}
Let $\{L_i\}_{i=1}^{u}$ be a family of hypersurfaces in $\mathbb P^n(\mathbb C)$ of the common degree $d$ and let $f$ be a meromorphic mapping of $\B^m(R_0)\subset\mathbb C^m$ into $\mathbb P^n(\mathbb C)$, where $u=\binom{n+d}{n}$. Assume that $\{L_i\}_{i=1}^{u}$ are linearly independent. Then, for every $0<r_0<r<R_0$, we have
$$ T_F(r,r_0)=dT_f(r,r_0)+O(1), $$
where $F$ is the meromorphic mapping of $\B^m(R_0)$ into $\P^{u-1}(\C)$ defined by the representation $F=(L_1(\tilde f):\cdots :L_{u}(\tilde f))$.
\end{lemma} 

\vskip0.2cm
\noindent
{\bf Proof of Theorem \ref{1.1}.} By using the universal covering if necessary, we may assume that $M=\B^{m}(1)$.

Replacing  $Q_j$ by $Q^{\frac{d}{d_j}}_j\ (j=1,...,q)$ if necessary, we may assume that $Q_j\ (j=1,\ldots ,q)$ have the same of the common degree $d$.

It is easy to see that there is a positive constant $\beta$ such that $\beta ||\tilde f||^d\ge |Q_i(\tilde f)|$ for every $1\le i\le q.$
We set 
$$ \mathcal A=\{(i_1,...,i_{k+1})\ ; 1\le i_j\le q, i_j\ne i_t\ \forall j\ne t\}. $$
For each $I=(i_1,...,i_{k+1})\in\mathcal A$, we denote by $P_{I1},...,P_{I(n+1)}$ the hypersurfaces obtained in Lemma \ref{lem3.1} with respect to the family of hypersurfaces $\{Q_{i_1},...,Q_{i_{k+1}}\}$. It is easy to see that there exists a positive constant $B\ge 1$, which is chosen common for all $I\in\mathcal A$, such that
$$ |P_{It}(\omega)|\le B\max_{1\le j\le k+1-n+t}|Q_{i_j}(\omega)|, $$
for all $\omega=(\omega_0,...,\omega_n)\in\C^{n+1}$.

Consider a reduced representation $\tilde f=(f_0,\ldots ,f_n): \B^m(1)\rightarrow \C^{n+1}$ of $f$. For a fixed point $z\in \B^m(1)\setminus\bigcup_{i=1}^qQ_i(\tilde f)^{-1}(\{0\})$.  We may assume that
$$ |Q_{i_1}(\tilde f)(z)|\le |Q_{i_2}(\tilde f)(z)|\le\cdots\le |Q_{i_{q}}(\tilde f)(z)|. $$
Since $Q_{i_1},\ldots,Q_{i_q}$ are in $k-$subgeneral position, by Lemma \ref{lem3.2}, there exists a positive constant $A$, which is chosen common for all $z$ and $(i_1,...,i_{q})$, such that
$$ ||\tilde f (z)||^d\le A\max_{1\le j\le k+1}|Q_{i_j}(\tilde f)(z)|=A|Q_{i_{k+1}}(\tilde f)(z)|.$$
Therefore, we have
\begin{align*}
\prod_{i=1}^q\dfrac{||\tilde f (z)||^d}{|Q_i(\tilde f)(z)|}&\le A^{q-k}\prod_{j=1}^{k}\dfrac{||\tilde f (z)||^d}{|Q_{i_j}(\tilde f)(z)|}\\
&\le A^{q-k}B^{n}\dfrac{||\tilde f (z)||^{kd}}{\bigl (\prod_{j=2}^{k-n+1}|Q_{i_j}(\tilde f)(z)|\bigl )\cdot\prod_{j=1}^{n}|P_{Ij}(\tilde f)(z)|}\\
&\le c_1\dfrac{||\tilde f (z)||^{(k-n+1)nd}}{\prod_{j=1}^{n}|P_{Ij}(\tilde f)(z)|^{k-n+1}},
\end{align*}
where $I=(i_1,...,i_{k+1})$ and $c_1$ is a positive constant, which is chosen common for all $I\in\mathcal A$.
The above inequality implies that
\begin{align}\label{3.2}
\log \prod_{i=1}^q\dfrac{||\tilde f (z)||^d}{|Q_i(\tilde f)(z)|}\le \log c_1+(k-n+1)\log \prod_{j=1}^n\dfrac{||\tilde f (z)||^{d}}{|P_{Ij}(\tilde f)(z)|}.
\end{align}

Now, for a positive integer $L$, we denote by $V_L$ the vector subspace of$\C[x_0,\ldots, x_n]$ which consists of  all homogeneous polynomials of degree $L$ and zero polynomial. We see that $N$ divisible by $d$. Hence, for each $(i)=(i_1,\ldots,i_n)\in\mathbf N_0^n$ with $\sigma (i)=\sum_{s=1}^ni_s\le\frac{N}{d}$, we set
$$W^{I}_{(i)}=\sum_{(j)=(j_1,\ldots ,j_n)\ge (i)}P_{I1}^{j_1}\cdots P_{In}^{j_n}\cdot V_{N-d\sigma (j)}.$$
Then we see that $W^{I}_{(0,\ldots,0)}=V_N$ and $W^{I}_{(i)}\supset W^{I}_{(j)}$ if $(i)<(j)$ (in the sense of lexicographic order). Therefore, $W^{I}_{(i)}$ is a filtration of $V_N$. 
We have the following lemma due to \cite{CZ}.
\begin{lemma}\label{lem3.2}
Let $(i)=(i_1,\ldots ,i_n),(i)'=(i_1',\ldots ,i_n')\in \mathbf N^n_0$. Suppose that $(i')$ follows $(i)$ in the lexicographic ordering and defined
$$ m^{I}_{(i)}=\dim \dfrac{W^{I}_{(i)}}{W^{I}_{(i)'}}.$$
Then, we have $m^I_{(i)}=d^n,$ provided $d \sigma (i)<N-nd$. 
\end{lemma}
We assume that 
$$ V_N=W^{I}_{(i)_1}\supset W^{I}_{(i)_2}\supset\cdots\supset W^{I}_{(i)_K}, $$
where $(i)_s=(i_{1s},...,i_{ns})$, $W^{I}_{(i)_{s+1}}$ follows $W^{I}_{(i)_s}$ in the ordering and $(i)_K=(\frac{N}{d},0,\ldots ,0)$. We see that $K$ is the number of $n$-tuples $(i_1,\ldots,i_n)$ with $i_j\ge 0$ and $i_1+\cdots +i_n\le\frac{N}{d}$. Then we easily estimate that
$$ K =\binom{\frac{N}{d}+n}{n}.$$
We define $m^{I}_s=\dim\frac{W^{I}_{(i)_s}}{W^{I}_{(i)_{s+1}}}$ for all $s=1,\ldots, K-1$ and set $m^{I}_K=1$. 

Let $u=\dim V_N$. From the above filtration, we may choose a basis $\{\psi^{I}_1,\ldots,\psi^{I}_{u}\}$ of $V_N$ such that  
$$\{\psi_{u-(m^{I}_s+\cdots +m^{I}_K)+1},\ldots ,\psi_{u}\}$$
 is a basis of $W^{I}_{(i)_s}$. For each $s\in\{1,\ldots,K\}$ and $l\in\{u-(m^{I}_s+\cdots +m^{I}_K)+1,\ldots, u-(m^{I}_{s+1}+\cdots +m^{I}_K)\}$, we may write
$$ \psi^{I}_l=P_{I1}^{i_{1s}}\ldots P_{In}^{i_{ns}}h_l,\ \text{ where } (i_{1s},\ldots,i_{ks})=(i)_s, h_l\in V^{I}_{N-d\sigma (i)_s}. $$
Then we have
\begin{align*}
|\psi^{I}_l(\tilde f)(z)|&\le |P_{I1} (\tilde f)(z)|^{i_{1s}}\ldots |P_{In} (\tilde f)(z)|^{i_{ks}}|h_l(\tilde f)(z)|\\
& \le c_2|P_{I1} (\tilde f)(z)|^{i_{1s}}\ldots |P_{In} (\tilde f)(z)|^{i_{ks}}||\tilde f(z)||^{N-d\sigma(i)_s}\\
&=c_2\left (\dfrac{|P_{I1} (\tilde f)(z)|}{||\tilde f(z)||^d}\right)^{i_{1s}}\ldots\left (\dfrac{|P_{In} (\tilde f)(z)|}{||\tilde f(z)||^d}\right)^{i_{ks}}||\tilde f (z)||^N,
\end{align*} 
where $c_2$ is a positive constant independently from $l$, $I$, $f$ and $z$. This implies that
\begin{align}\label{3.7}
\begin{split}
\log\prod_{l=1}^{u}|\psi^{I}_l(\tilde f)(z)|&\le\sum_{s=1}^Km^{I}_s\left (i_{1s}\log\dfrac{|P_{I1} (\tilde f)(z)|}{||\tilde f(z)||^d}+\cdots+i_{ns}\log\dfrac{|P_{In} (\tilde f)(z)|}{||\tilde f(z)||^d}\right)\\ 
& \ \ \ +uN\log ||\tilde f (z)||+\log c_2.
\end{split}
\end{align}

We fix $\phi_1,...,\phi_{u}$, a basic of $V_N$, $\psi_s^{I}(\tilde f)=L_s^{I}(\tilde F)$, where $L_s^{I}$ are linear forms and $\tilde F=(\phi_1(\tilde f),\ldots ,\phi_{u}(\tilde f))$ is a reduced representation of a meromorphic mapping $F$. We set
\begin{align*}
b_j^{I}=\sum_{s=1}^{K}m_s^{I}i_{js},\ 1\le j\le k.
\end{align*}
From (\ref{3.7}) we have that
$$\log\prod_{s=1}^{u}|L_{s}^{I}(\tilde F)(z)|\le\log\left(\prod_{j=1}^{n}\biggl (\dfrac{|P_{Ij}(\tilde f)(z)|}{||\tilde f(z)||^d}\biggl )^{b_j^{I}}\right)+uN\log ||\tilde f(z)||+\log c_2.$$
We set $b=\min_{j,I}b_j^{I}$. Because $f$ is algebraically non degenerate over $\C$, $F$ is linearly non degenerate over $\C$. Then there exists an admissible set $\alpha =(\alpha_1,...,\alpha_{u})\in (\mathbb Z_+^m)^u$, with $|\alpha_s|\le s-1$, such that
\begin{align*}
W^{\alpha}(\phi_s(\tilde f)):=\det (\mathcal D^{\alpha_i}(\phi_s(\tilde f)))_{1\le i,s\le u}\not\equiv 0.
\end{align*}
We also have
\begin{align}\label{aux}
\begin{split}
\log \frac{||\tilde f(z)||^{qdb}|W^{\alpha}(\phi_s(\tilde f))(z)|^p}{\prod_{i=1}^{q}|Q_i(\tilde f)(z)|^b}&\le \log\frac{||\tilde f(z)||^{pndb}|W^\alpha(\phi_s(\tilde f))(z)|^p}{\prod_{j=1}^n|P_{Ij}(\tilde f)(z)|^{pb}}+O(1)\\
&\le\log \frac{||\tilde f(z)||^{pd\sum_{j=1}^{n}b_j^{I}}|W^{\alpha}(\phi_s(\tilde f))(z)|^p}{\prod_{j=1}^{n}|P_{Ij}(\tilde f)(z)|^{pb^{I}_j}} +O(1)\\
&\le \log \frac{||\tilde f(z)||^{puN}|W^{\alpha}(\phi_s(\tilde f))(z)|^p}{\prod_{i=1}^{u}|\psi_i^{I}(\tilde f)(z)|^p}+O(1)\\
&\le \log \frac{||\tilde f(z)||^{puN}|W^{\alpha}(\psi^I_s(\tilde f))(z)|^p}{\prod_{i=1}^{u}|\psi_i^{I}(\tilde f)(z)|^p}+O(1),
\end{split}
\end{align}
where $W^{\alpha}(\psi^I_s(\tilde f))=\det (\mathcal D^{\alpha_i}(\psi^I_{s}(\tilde f)))_{1\le i,s\le u}$, $O(1)$ depends only on $N$ and $\{Q_i\}_{i=1}^{q}$.
This inequality implies that
\begin{align}\label{3.6}
\log \frac{||\tilde f(z)||^{qdb-puN}|W^{\alpha}(\phi_s(\tilde f))(z)|^p}{(\prod_{i=1}^{q}|Q_i(\tilde f)(z)|^b)}\le\log\frac{|W^{\alpha}(\phi_s(\tilde f))(z)|^p}{\prod_{i=1}^{u}|\psi_i^{I}(\tilde f)(z)|^p}+O(1),
\end{align}
for all $z\in\C^m$ outside a proper analytic subset of $\C^m$, which is the union of zero sets of functions $Q_i(\tilde f), P_{Ij}(\tilde f)$.

Put $S_{I}=\dfrac{|W^{\alpha}(\phi_s(\tilde f))(z)|}{\prod_{i=1}^{u}|\psi_i^{I}(\tilde f)(z)|}$. Then, there exists a positive constant $K_0$ such that, for each  $z\in \mathbb C^m$,
\begin{align*}
\frac{||\tilde f(z)||^{qdb-puN}|W^{\alpha}(\phi_s(\tilde f))(z)|^p}{\prod_{i=1}^{q}|Q_i(\tilde f)(z)|^b}\le K^p_0.S^p_{I}(z).
\end{align*}
for some $I\subset \{1,...,q\}$ with $\sharp I=k+1$. 

\begin{lemma}\label{new} For $N=(n+1)d+p(n+1)^3I(\epsilon^{-1})$ as in the assumption,
we have 
\begin{align*}
\mathrm{(a)}&\ \ \ \frac{puN}{db}\le (k-n+1)(n+1)+\epsilon,\\
\mathrm{(b)}&\ \ \  u\le e^{n+2}\left (dp(n+1)^2I(\epsilon^{-1})\right )^n.
\end{align*}
\end{lemma}

\noindent
\textit{Proof of Lemma.} For a real number $x\in[0,\frac{1}{(n+1)^2}]$, we have
\begin{align}\label{new1.2}
\begin{split}
(1+x)^{n}&=1+nx+\sum_{i=2}^n\binom{n}{i}x^{i}\le 1+nx+\sum_{i=1}^2\frac{n^{i}}{i!(n+1)^{2i-2}}x\\
&=1+nx+\sum_{i=2}^n\frac{1}{i!}x\le 1+(n+1)x.
\end{split}
\end{align}
We also note that
\begin{align}\label{new1.3}
\frac{(n+1)d}{N-(n+1)d}=\frac{(n+1)d}{p(n+1)^3I(\epsilon^{-1})d}\le \frac{1}{(n+1)^2}.
\end{align}

Now, we have the following estimates. First, 
$$u=\binom{N+n}{n}=\frac{(N+1)\cdots (N+n)}{1\cdots n}.$$
Second, since the number of nonnegative integer $t$-tuples with summation $\le T$ is equal to the number of nonnegative integer $(t+1)$-tuples with summation exactly equal $T\in \Z$, which is $\bigl (^{T+t}_{\ \ t}\bigl )$, since the sum below is independent of $j$, we have that
\begin{align*}
b_{j}^{I}=&\sum_{\sigma (i)\le N/d}m^I_{(i)}i_{j}\ge\sum_{\sigma (i)\le N/d-n}m^I_{(i)}i_{j}\\ 
=&\sum_{\sigma (i)\le N/d-n}d^ni_j=\frac{d^n}{n+1}\sum_{\sigma (i)\le N/d-n}\sum_{j=1}^{n+1}i_j\\
=&\frac{d^n}{n+1}\sum_{\sigma (i)\le N/d-n}\left (\frac{N}{d}-n\right )=\frac{d^n}{(n+1)}\left (\frac{N}{d}-n\right)\binom{N/d}{n}\\
&=\frac{d^n(N/d)(N/d-1)\cdots (N/d-n-1)(N/d-n)}{1\cdots (n+1)d}\\
&=\frac{N(N-d)\cdots (N-(n-1)d)(N-nd)}{(n+1)!d}.
\end{align*}
This implies that 
\begin{align*}
\frac{puN}{db}\le &p(n+1)\frac{(N+1)\cdots (N+n)}{(N-d)\cdots (N-nd)}=p(n+1)\prod_{j=1}^n\frac{N+j}{N-(n+1)d+jd}\\
\le & p(n+1)\left (\frac{N}{N-(n+1)d}\right )^n\le p(n+1)\left (1+\frac{(n+1)d}{N-(n+1)d}\right )^n\\
\le &p(n+1)\left (1+(n+1)\frac{(n+1)d}{N-(n+1)d}\right)\ \ \ (*)\\
\le &p(n+1)\left (1+(n+1)\frac{(n+1)d}{p(n+1)^3I(\epsilon^{-1})d}\right )\\
\le &p(n+1)\left (1+\frac{1}{p(n+1)\epsilon^{-1}}\right )=p(n+1)+\epsilon,
\end{align*}
where the inequality (*) comes from (\ref{new1.2}) and (\ref{new1.3}).
Also, one can be estimated that
\begin{align*}
u&=\binom{N+n}{n}\le e^n\left (1+\frac{N}{n}\right )^n\le e^n\left (\frac{n+(n+1)d}{n}+\frac{p(n+1)^3I(\epsilon^{-1})d}{n}\right )^n\\
&=e^n(p(n+1)^2I(\epsilon^{-1})d)^n\left (1+\frac{1}{n}+\frac{n+(n+1)d}{np(n+1)^2I(\epsilon^{-1})d}\right)^n\\
&\le \left (edp(n+1)^3I(\epsilon^{-1})\right )^n\cdot\left (1+\frac{1}{n}+\frac{2}{n(n+1)}\right)^n\\
&\le \left (edp(n+1)^3I(\epsilon^{-1})\right )^n\cdot\left (1+\frac{2}{n}\right)^n\le e^{n+2}\left (dp(n+1)^2I(\epsilon^{-1})\right )^n.
\end{align*}
The lemma is proved.

\begin{claim}\label{cl3.10}
$\bigl (b\sum_{j=1}^q\nu_{Q_j(\tilde f)}-p\nu_{W^{\alpha}(\phi_s(\tilde f))}\bigl )\le b\sum_{i=1}^q\min\{u-1,\nu_{Q_j(\tilde f)}\}$.
\end{claim}
Fix $z\in \C^m$, we may assume that
$$\nu_{Q_1(\tilde f)}(z)\ge\cdots
\ge\nu_{Q_{t}(\tilde f)}(z)>0=\nu_{Q_{t+1}(\tilde f)}(z)=\cdots =\nu_{Q_{q}(\tilde f)}(z),$$
where $0\le t\le k,$ ($t$ may be zero). We denote by $\{P_1,\ldots ,P_{n+1}\}$, the family of hypersurfaces corresponding to the family $\{Q_1,...,Q_{k+1}\}$ as in the Lemma \ref{lem3.1}. Then we will see that
\begin{align*}
\nu_{P_1}(z)&=\nu_{Q_1}(z),\\ 
\nu_{P_i}(z)&\ge\nu_{Q_{k-n+i}}(z). 
\end{align*}
Put $I =(1,...,n+1)$ and $M=u-1$. We have
$$
p\nu_{W^{\alpha}(\phi_s(\tilde f))}(z)=p\nu_{W^\alpha(\psi_s^{I}(\tilde f))}(z)
\ge p\sum_{s=1}^{u}\max\{\nu_{\psi_s^{I}(\tilde f)}(z)-M,0\}.
$$
For $\psi=P_1^{i_1}...P_n^{i_n}h\in\{\psi_s^{I}\}_{s=1}^{u}$, we have
$$
\psi (\tilde f)(z)=P_{1}^{i_1}(\tilde f)(z)\ldots P_n^{i_n}(\tilde f)(z).h(\tilde f)(z).
$$
Hence
\begin{align*}
\max \{ \nu_{\psi (\tilde f)}(z)-M,0\}
\ge & \sum_{t=1}^{n}\max \{\nu_{(P_t^{i_t}(\tilde f)}(z)-M,0 \}\\
\ge & \sum_{t=1}^{n}i_t\max \{\nu_{P_t(\tilde f)}(z)-M,0\}.
\end{align*}
This implies that
\begin{align*}
p\sum_{s=1}^{u}\max&\{\nu_{\psi_s(\tilde f)}(z)-M,0\}\ge p\sum_{(i)}m_{(i)}^{I}\sum_{t=1}^{k}i_t\max\{\nu_{P_t(\tilde f)}(z)-M,0\}\\
&=p\sum_{t=1}^{n}b_t^{I}\max\{\nu_{P_t(\tilde f)}(z)-M,0\}\ge p\sum_{t=1}^nb\max\{\nu_{P_t(\tilde f)}(z)-M,0\}\\
&\ge\sum_{t=1}^{K}b\max\{\nu_{Q_t(\tilde f)}(z)-M,0\}=\sum_{i=1}^{q}b\max\{\nu_{Q_i(\tilde f)}(z)-M,0\}\\
&=b\sum_{i=1}^{q}\max\{(f^*Q_i)(z)-M,0\}=b\sum_{i=1}^{q}\bigl (\nu_{Q_i(\tilde f)}(z)-\min\{u-1,\nu_{Q_i(\tilde f)}(z)\}\bigl ).
\end{align*}
Hence 
\begin{align*}
b\sum_{i=1}^{q}\nu_{Q_i(\tilde f)}(z)-p\nu_{W^{\alpha}(\phi_i(\tilde f))}(z)&\le b\sum_{i=1}^{q}\min\{u-1,\nu_{Q_i(\tilde f)}\}.
\end{align*}
The claim is proved.

Assume that 
$$ \rho\Omega_f+\dfrac{\sqrt{-1}}{2\pi}\partial\bar\partial\log h^2\ge \ric\omega.$$
We now suppose that
$$ \sum_{j=1}^q\delta^{[u-1]}_f(Q_j)> \frac{puN}{db}+\dfrac{\rho pu(u-1)}{db}.$$
Then, for each $j\in\{1,\ldots ,q\},$ there exist constants $\eta_j>0$ and continuous plurisubharmonic function $\tilde u_j$ such that 
$e^{\tilde u_j}|\varphi_j|\le ||\tilde f||^{d\eta_j},$ where $\varphi_j$ is a holomorphic function with $\nu_{\varphi_j}=\min\{u-1,f^*Q_j\}$ and
$$ q-\sum_{j=1}^q\eta_j>  \frac{puN}{db}+\dfrac{\rho pu(u-1)}{db}.$$
Put $u_j=\tilde u_j+\log |\varphi_j|$, then $u_j$ is a plurisubharmonic and
$$ e^{u_j}\le ||\tilde f||^{d\eta_j},\ j=1,\ldots ,q. $$
Let
$$v (z)=\log\left |(z^{\alpha_1+\cdots+\alpha_u})^p\dfrac{(W^{\alpha}(\phi_s(\tilde f))(z))^p}{(\prod_{i=1}^{q}Q_i(\tilde f)(z))^b}\right |+b\sum_{j=1}^q u_j(z).$$
Therefore, we have the following current inequality
\begin{align*}
2dd^c[v]&\ge p[\nu_{W^{\alpha}(\phi_i(\tilde f))}]-b\sum_{j=1}^q[\nu_{Q_i(\tilde f)}]+\sum_{j=1}^q2dd^c[u_j]\\
&=p[\nu_{W^{\alpha}(\phi_i(\tilde f))}]-b\sum_{j=1}^q[\nu_{Q_i(\tilde f)}]+b\sum_{j=1}^q[\min\{u-1,\nu_{Q_i(\tilde f)}\}]\ge 0.
\end{align*}
This implies that $v$ is a plurisubharmonic function on $\B^m(1)$.

On the other hand, by the growth condition of $f$, there exists a continuous plurisubharmonic function $\omega\not\equiv\infty$ on $\B^m(1)$ such that
\begin{align*}
e^\omega dV\le ||\tilde f||^{2\rho}v_m
\end{align*}
Set
$$t=\dfrac{2\rho}{db(q-\frac{puN}{db}-\sum_{j=1}^q\eta_j)}>0$$ 
and 
$$\lambda (z)=(z^{\alpha_1+\cdots +\alpha_u})^p\dfrac{\left (W^{\alpha}(\phi_i(\tilde f))\right )^p(z)}{Q_1^{b}(\tilde f)(z)\ldots Q_q^{b}(\tilde f)(z)}.$$ 
We see that
$$ \frac{u(u-1)p}{2}t< \frac{u(u-1)p}{2}\cdot\dfrac{2\rho}{2\rho u(u-1)p}=1,$$
and the function $\zeta=\omega+ tv$ is plurisubharmonic on the K\"{a}hler manifold $M$. Choose a position number $\delta$ such that $0<\frac{u(u-1)pt}{2}<\delta<1.$
Then, we have
\begin{align*}
e^\zeta dV&=e^{\omega +tv}dV\le e^{tv}||\tilde f||^{2\rho}v_m=|\lambda|^{t}(\prod_{j=1}^qe^{tbu_j})||\tilde f||^{2\rho}v_m\\
&\le|\lambda|^t ||\tilde f||^{2\rho+\sum_{j=1}^qbdt\eta_j}v_m=|\lambda|^t ||\tilde f||^{dbt(q- \frac{puN}{db})}v_m.
\end{align*}
Integrating both sides of the above inequality over $\B^m(1),$  we have
\begin{align}\label{3.11}
\begin{split}
\int_{\B^m(1)}e^\zeta dV&\le \int_{\B^m(1)}|\lambda|^t ||\tilde f||^{t(qdb- puN)}v_m.\\
&=2m\int_0^1r^{2m-1}\left (\int_{S(r)}\bigl (|\lambda| ||\tilde f||^{qdb- puN}\bigl )^t\sigma_m\right )dr\\
&\le 2m\int_0^1r^{2m-1}\left (\int_{S(r)}\sum_{\overset{\sharp I=k+1}{I\subset\{1,...,q\}}}\bigl |(z^{\alpha_1+\cdots +\alpha_u})K_0S_{I}\bigl |^{pt}\sigma_m\right )dr.
\end{split}
\end{align}

(a) We first consider the case where
$$ \lim\limits_{r\rightarrow 1}\sup\dfrac{T_f(r,r_0)}{\log 1/(1-r)}<\infty.$$
We note that $(\sum_{i=1}^u|\alpha_i|)pt\le \frac{u(u-1)p}{2}t<\delta<1$. Then by Proposition \ref{pro2.2}, there exists a positive constant $K_1$ such that, for every $0<r_0<r<r'<1,$ we have
\begin{align*}
\int_{S(r)}\left |(z^{\alpha_1+\cdots +\alpha_u})K_0S_I(z)\right |^{pt}\sigma_m\le K_1\left (\dfrac{r'^{2m-1}}{r'-r}dT_f(r',r_0)\right )^{\delta}.
\end{align*}
Choosing $r'=r+\dfrac{1-r}{eT_f(r,r_0)}$, we get
$$ T_f(r',r_0)\le 2T_f(r,r_0)$$
outside a subset $E\subset [0,1]$ with $\int_E\frac{dr}{1-r}<+\infty$. Hence, the above inequality implies that
\begin{align*}
\sum_{\overset{\sharp I=k+1}{I\subset\{1,...,q\}}}\int_{S(r)}\left |(z^{\alpha_1+\cdots +\alpha_u})K_0S_I(z)\right |^{pt}\sigma_m\le \dfrac{K}{(1-r)^\delta}\left (\log\frac{1}{1-r}\right )^{\delta}
\end{align*}
for all $z$ outside $E$, where $K$ is a some positive constant. By choosing $K$ large enough, we may assume that the above inequality holds for all $z\in \B^m(1)$.
Then, the inequality (\ref{3.11}) yields that
\begin{align*}
\int_{\B^m(1)}e^\zeta dV&\le 2m\int_0^1r^{2m-1}\dfrac{K}{(1-r)^\delta}\left (\log\frac{1}{1-r}\right )^{\delta}dr< +\infty
\end{align*}
This contradicts the results of S.T. Yau \cite{Y76} and L. Karp \cite{K82}. 

Hence, we must have
$$\sum_{j=1}^q\delta^{[u-1]}_{f}(Q_j)\le \frac{puN}{db}+\dfrac{p\rho u(u-1)}{db}.$$
Since $p\le b$, the above inequality implies that
$$\sum_{j=1}^q\delta^{[u-1]}_{f}(Q_j)\le (k-n+1)(n+1)+\epsilon+\dfrac{\rho u(u-1)}{d}.  $$
The theorem is proved in this case.

(b) We now consider the remaining case where 
$$ \lim\limits_{r\rightarrow 1}\sup\dfrac{T(r,r_0)}{\log 1/(1-r)}= \infty .$$
Repeating the argument in the proof of Theorem \ref{1.1}, we only need to prove the following theorem.
\begin{theorem}\label{5.3}
With the assumption of Theorem \ref{1.1} and suppose that $M=\B^m(R_0)$.
 Then, we have
$$(q-p(n+1)-\epsilon)T_f(r,r_0)\le \sum_{i=1}^{q}\dfrac{1}{d}N^{[u-1]}_{Q_i(\tilde f)}(r)+S(r),$$
where $S(r)\le K(\log^+\frac{1}{R_0-r}+\log^+T_f(r,r_0))$
for all $0<r_0<r<R_0$ outside a set $E\subset [0,R_0]$ with $\int_E\frac{dt}{R_0-t}<\infty.$ 
\end{theorem}
\begin{proof}
Repeating the above argument, we have
$$\int_{S(r)}\left |(z^{\alpha_1+\cdots +\alpha_u})^p\frac{||\tilde f(z)||^{qdb-puN}|W^{\alpha}(\phi_s(\tilde f))(z)|^p}{\prod_{i=1}^{q}|Q_i(\tilde f)(z)|^b}\right |^{t}\sigma_m\le K_1\left (\dfrac{R^{2m-1}}{R-r}dT_f(R,r_0)\right )^{\delta}$$
for every $0<r_0<r<R<R_0$. Using the concativity of the logarithmic function, we have
\begin{align}\label{5.4}
\begin{split}
p\int_{S(r)}&\log |(z^{\alpha_1+\cdots +\alpha_u})|\sigma_m+(qdb-puN)\int_{S(r)}\log ||\tilde f||\sigma_m+p\int_{S(r)}\log |W^{\alpha}(\phi_s(\tilde f))|\sigma_m\\ 
&-b\sum_{j=1}^q \int_{S(r)}\log |Q_j(\tilde f)|\sigma_m\le K\left (\log^+\dfrac{1}{R_0-r}+\log^+T_f(R,r_0)\right )
\end{split}
\end{align}
for some positive constant $K$. By the Jensen formula, this inequality implies that
\begin{align}\label{3.14}
\begin{split}
(qdb-puN)T_f(r,r_0)&+pN_{W^{\alpha}(\phi_s(\tilde f))}(r)-b\sum_{i=1}^qN_{Q_i(\tilde f)}(r)\\
&\le K\left (\log^+\dfrac{1}{R_0-r}+\log^+T_f(R,r_0)\right )+O(1).
\end{split}
\end{align}
From Claim \ref{cl3.10}, we have 
$$ b\sum_{i=1}^qN_{Q_i(\tilde f)}(r)-pN_{W^{\alpha}(\phi_s(\tilde f))}(r)\le\sum_{i=1}^qN^{[u-1]}_{Q_i(\tilde f)}(r).$$
Combining this estimate and (\ref{3.14}), we get
\begin{align*}
\left (q-\frac{puN}{db}\right)T_f(r,r_0)&\le\sum_{i=1}^{q}\frac{1}{d}N^{[u-1]}_{Q_i(\tilde f)}(r)+K\left (\log^+\dfrac{1}{R_0-r}+\log^+T_f(R,r_0)\right )+O(1).
\end{align*}
Since $\frac{puN}{db}\le p(n+1)+\epsilon$, the above inequality  implies that
\begin{align*}
(q-p(n+1)-\epsilon) T_f(r,r_0)&\le\sum_{i=1}^{q}\frac{1}{d}N^{[u-1]}_{Q_i(\tilde f)}(r)+K\left (\log^+\dfrac{1}{R_0-r}+\log^+T_f(R,r_0)\right )+O(1).
\end{align*}
Choosing $R=r+\dfrac{1-r}{eT_f(r,r_0)}$, we get
$$ T_f(R,r_0)\le 2T_f(r,r_0)$$
outside a subset $E\subset [0,1]$ with $\int_E\frac{dr}{1-r}<+\infty$. 
Thus
$$ (q-p(n+1)-\epsilon) T_f(r,r_0)\le\sum_{i=1}^{q}\frac{1}{d}N^{[u-1]}_{Q_i(\tilde f)}(r)+K\left (\log^+\dfrac{1}{R_0-r}+\log^+T(r,r_0)\right )+O(1). $$
This implies that
$$\sum_{j=1}^q\delta^{[u-1]}_{f}(Q_j)\le\sum_{j=1}^q\delta^{[u-1]}_{f,*}(Q_j)\le p(n+1)+\epsilon.$$
The theorem is proved in this case.
\end{proof}

\section{Value distribution of the Gauss map of a complete regular submanifold of $\C^m$}

Let $M$ be a connected complex manifold of dimension $m$. Let 
$$f = (f_1,\ldots , f_n) : M \rightarrow \C^n$$
be a regular submanifold of $\C^n$; namely, $f$ be a holomorphic map of $M$ into $\C^n$ such that $\rank d_pf = \dim M$ for every point 
$p\in M.$ We assign each point $p\in M$ to the tangent space $T_p(M)$ of $M$ at $p$ which may be considered as an $m$-dimensional linear subspace of $T_{f(p)}(\C^n)$. Also, each tangent space $T_p(\C^n)$ can be identified with $T_0(\C^n)= \C^n$ by a parallel translation. Hence, each $T_p(M)$ is corresponded to a point $G(p)$ in the complex Grassmannian manifold $G(m,n)$ of all $m$-dimensional linear subspaces of $\C^n$.

\begin{definition}
The map $G : p\in M \mapsto G(p)\in G(m,n)$ is called the Gauss map of the map $f : M \rightarrow \C^n$. 
\end{definition}
The space $G(m,n)$ is canonically embedded in $\P^N(\C)=\P(\bigwedge^m\C^n)$, where $N =\binom{n}{m}-1$. Then we may identify the Gauss map $G$ with a holomorphic mapping of $M$ into $\P^N(\C)$ given as follows: taking holomorphic local coordinates $(z_1,\ldots ,z_n)$ defined on an open set $U$, we consider the map 
$$\bigwedge := D_1f\wedge\cdots\wedge D_nf: U\rightarrow\bigwedge^m\C^n\setminus\{0\},$$
where 
$$D_if = (\dfrac{\partial f_1}{\partial z_i},\cdots ,\dfrac{\partial f_n}{\partial z_i}).$$
Then, locally we have
$$G = \pi\circ\bigwedge,$$
where $\pi : \C^{N+1} \setminus\{0\}\rightarrow\P^N(\C)$ is the canonical projection map. A regular submanifold $M$ of $\C^m$ is considered 
as a K\"{a}hler manifold with the metric $\omega$ induced from the standard flat metric on $\C^m$. We denote by $dV$ the volume form on $M$. For arbitrarily holomorphic coordinates $z_1,\ldots ,z_m,$ we see that
$$dV =|\bigwedge|^2\left (\sqrt{-1}{2}\right )^mdz_1\wedge d\bar{z_1}\wedge\cdots\wedge dz_m\wedge dz_m,$$ 
where
$$|\bigwedge|^2=\sum_{1\le i_1<\cdots <i_m\le n}\dfrac{\partial(f_{i_1},...,f_{i_m})}{\partial(z_1,...,z_m)}^2.$$
Therefore, for a regular submanifold $f : M \rightarrow \C^m$, the Gauss map $G : M \rightarrow \P^N(\C)$ satisfies the following growth condition
$$\Omega_G + dd^c \log h^2 = dd^c \log |\bigwedge|^2 = \ric (\omega),$$
where $h = 1$.
Then Theorem \ref{1.1} immediately gives us the following.
\begin{theorem}\label{6.1.}
Let $M$ be a complex manifold of dimension $m$ such that the universal covering of $M$ is biholomorphic to a ball $\B^m(R_0)\ (0< R_0\le +\infty)$ 
in $\C^m.$ Let $f :M\rightarrow\C^n$ be a complete regular submanifold.  Assume that the Gauss map $G: M\rightarrow \P^N(\C)$ is algebraically non-degenerate, where $N=\binom{n}{m}-1$. Let $Q_1,\ldots ,Q_q$ be $q$ hypersurfaces of degree $d_j \ (1\leq j\leq q)$ in $k$-subgeneral position in $\P^N(\C)$. Let $d$ be the least common multiple of $d_i$'s, i.e., $d =l.c.m.\{d_1,\ldots ,d_q\}$. Then, for every $\epsilon>0$ we have
$$\sum_{i=1}^q\delta_G^{[u-1]}(Q_i) \le p(N+1)+\epsilon+\frac{\rho u(u-1)}{d},$$
where $p=k-N+1$, $L=(N+1)d+p(N+1)^3I(\epsilon^{-1})d$ and $u=\binom{L+N}{N}\le 3^{N+2}(dp(N+1)^2I(\epsilon^{-1}))^N$.
\end{theorem}

{\bf Acknowledgements.} This research is funded by Vietnam National Foundation for Science and Technology Development (NAFOSTED) under grant number 101.04-2015.03.

\vskip0.2cm
{\footnotesize 
\noindent
{\sc Si Duc Quang}\\
$^1$ Department of Mathematics, Hanoi National University of Education,\\
 136-Xuan Thuy, Cau Giay, Hanoi, Vietnam.\\
$^2$ Thang Long Institute of Mathematics and Applied Sciences,\\
 Nghiem Xuan Yem, Hoang Mai, HaNoi, Vietnam.\\
\textit{E-mail}: quangsd@hnue.edu.vn}

\vskip0.2cm
{\footnotesize 
\noindent
{\sc Nguyen Thi Quynh Phuong}\\
Department of Mathematics, Hanoi National University of Education, \\
136-Xuan Thuy, Cau Giay, Hanoi, Vietnam.\\
\textit{E-mail}: nguyenthiquynhphuongk63@gmail.com}

\vskip0.2cm
{\footnotesize 
\noindent
{\sc Nguyen Thi Nhung}\\
Department of Mathematics, Thang Long University,\\
Nghiem Xuan Yem, Hoang Mai, HaNoi, Vietnam.\\
\textit{E-mail}: hoangnhung227@gmail.com}

\end{document}